\documentclass{cimart}

\usepackage{psfrag}
\usepackage{subfigure}

\newcommand{\N}{{\mathbb{N}}}
\newcommand{\Of}{\mathcal{O}_f}
\newcommand{\R}{{\mathbb{R}}}
\newcommand{\T}{{\mathbb{T}}}
\newcommand{\Z}{{\mathbb{Z}}}

\DeclareMathOperator{\Deck}{Deck}

\DeclareMathOperator{\fix}{Fix}
\DeclareMathOperator{\id}{Id}
\DeclareMathOperator{\per}{Per}

\title[A survey on the growth rate inequality for sphere endomorphisms]{A survey on the growth rate inequality for sphere\\endomorphisms}

\authors{Juliana Xavier}

\authorinfo{Universidad de la Rep\'ublica, Montevideo, Uruguay}{jxavier@fing.edu.uy}

\abstract{%
    We survey recent results and current challenges concerning the growth rate inequality for sphere endomorphisms, and present a number of open problems and conjectures arising in this context.
    }


\VOLUME{34}
\YEAR{2026}
\ISSUE{2}
\NUMBER{11}
\DOI{https://doi.org/10.46298/cm.17178}
\editinfo{December 23, 2025}{March 23, 2026}{Tiago Macedo, Jaqueline Mesquita, Rafael Potrie, and Mariel Saez}

\begin{document}

\section{Introduction}

The growth rate inequality

\begin{equation}\label{tasa}
 \limsup_{n\to\infty} \frac{1}{n} \log (\# \fix (f^n))\geq \log (d)
\end{equation}

For a dynamical system $f$ of positive degree $d$, the inequality is both simple and natural. For instance, if
\[
f:S^1\to S^1
\]
is continuous with degree $d$, the inequality is easily seen to be true. Indeed, the graph of a lift
\[
\tilde f:\R\to \R
\]
must intersect at least $d-1$ times the set
\[
\{y=x+n : n\in \Z\}.
\]
This yields at least $d^n-1$ fixed points for the map $f^n$.

However, in contrast, one can easily construct a continuous map
\[
f:S^2\to S^2
\]
of degree $d$ for which inequality \eqref{tasa} does not hold. In polar coordinates on $\R^2$, consider
\[
(r,\theta)\mapsto (dr,d\theta).
\]
Extending this map to the sphere by setting
\[
\infty \mapsto \infty,
\]
we obtain a map of degree $d$ with only two periodic points.

More interestingly, to this day there is no well-understood list of obstructions for the inequality to hold on the \textit{two-dimensional sphere}. Moreover, in trying to understand this simple yet fundamental question in low-dimensional dynamics, several other significant open problems arise.
 
When an endomorphism $f$ satisfies inequality \eqref{tasa}, we say that $f$ \textit{has the rate}.

This paper provides an introductory survey of the basic results obtained to date on this ``rate'' problem, and presents several related open questions and directions for further research.

To finish this introductory section, I would like to add that it was Mike Shub that grew awareness around the growth rate inequality problem within the endomorphism group working 
in Montevideo, Uruguay, during one of his many stays in our country.  We are deeply thankful for his contribution, which resulted in many years of both joy and challenge in our work.

\section{\texorpdfstring{$C^1$}{C1} scenario}

Although the growth rate inequality \eqref{tasa} is topological in nature, its historical roots lie within differentiable dynamics.  To borrow the precise cadence of the original work, 
below is the first sentence in \cite{ps}:

``The relationship between the long term dynamics of an endomorphism of a manifold and its long term effect on the algebraic topology of the
manifold can depend on the smoothness of the endomorphism.''

We have already seen that if $f$ is a continuous map, the inequality \eqref{tasa} may not hold (i.e. $f$ may not have the rate). On the other hand, if $f$ is a rational map, then
it has the rate (see \cite[Proposition 1]{s}). From \cite{ss} it follows that if $f$ is $C^1$, then it must have infinitely many distinct periodic points, but their growth rate remains unknown.  Of course this raises
the question as to whether certain regularity will imply that the map has the rate.  This is not the focus of the present paper.  Nonetheless, we present in this section a
very brief overview of the results regarding differentiable maps.

The following is stated in \cite[Theorem 1]{ps}:

\begin{theorem}\label{1ps} If $f:S^2\to S^2$ is a $C^1$ latitude-preserving endomorphism of degree $2$, then for each $n$, $f^n$ has at least $2^n$ fixed points.
\end{theorem}

The latitude preserving hypothesis means that $f$ carries each latitude into another latitude or to one of the poles. It need not
be a homeomorphism from one latitude to another.  Of course, it follows from Theorem \ref{1ps} that $f$ has the rate.

This result was extended and further explained in \cite{mis}.

These sets of hypotheses (smooth and preserving
some ``geographical'' (singular) foliation) were later considered in \cite{gmnp}, where the authors considered what
happens when one changes the ``geography'' and replaces the assumption of smoothness by another one.  More precisely, certain foliations with only one singularity are specified and
the map is required to be smooth near this ``pole'' in order to have the rate.

The research in \cite{gmnp} was carried out by studying ``toy models'' which may help one to understand
the mechanism of creating periodic points under the assumption of smoothness.  See \cite{gmnp2}, where the authors confirm
Shub's conjecture for smooth maps of $S^2$ preserving the longitudinal foliation.

These strong hypotheses of preserving a foliation and being of class $C^1$ were 
dropped and replaced by an assumption on the homological action of the map that generalizes the former series of results:

\begin{theorem}[\!\!\cite{nancys}] Let $f:S^2\to S^2$ be a continuous map such that $\deg f = d, |d|>1$. Assume the following hypotheses hold:
$(H1)$ $f$ has two attracting fixed points denoted by $N$ and $S$; let $A=S^2\setminus \{N,S\}$,
$(H2)$ if a loop $\gamma\subset A\cap f^{-1}(A)$ is homotopically trivial in $A$, then $f(\gamma)$ is also homotopically trivial in $A$.
Then $f$ has the rate.
\end{theorem}

As an immediate result we get the following:

\begin{corollary} 
Let $f : S^2 \to S^2$ be a $C^1$ endomorphism of degree $d$ and assume that $|d| > 1$. If
$f^{-1}(\{N, S\}) \subset\{N, S\}$, then $f$ has the rate.
\end{corollary}

This result was re-obtained in \cite{barcor} also as a corollary of their work on the growth rate inequality for higher dimensional spheres.

A crucial remark, that is responsible for these last results and that will be exploited in the more detailed discussion to follow, is that the $C^1$ hypothesis is only used to guarantee 
that the map is a sink around its critical points. It is natural then not to impose differentiability, but rather try to  understand what the real obstructions are for a sphere 
endomorphism to have the rate.

For instance, we have:

\begin{theorem}[\!\!\cite{lhc}] Let $f:S^2\to S^2$
be a degree $d$ map such that all its periodic orbits are isolated
as invariant sets. Then, if $f$ has no sources of degree $r$ with $|r| > 1$ (this happens, in particular,
if $f$ is $C^1$) we have that $f$ has the rate.
\end{theorem}

\section{Annulus maps}\label{ac}

Sphere endomorphisms are too large of a universe to handle. It is natural to focus on the simplest cases at the beginning.  Other than homeomorphisms, the simplest maps one can think of, 
in terms of their local properties, are covering maps. A \textit{covering map} $f:X\to Y$ is a continuous map such that every $y\in Y$ has an open neighborhood $V$ whose 
preimage is a disjoint union of  open sets and  each of them is mapped homeomorphically onto $V$. If, however,  one is interested in dynamics, one needs $X$ to be equal to $Y$, that
is, one 
wants a \textit{self-covering} $f:X\to X$ in order to be able to iterate the map. However, the only possible covering maps $f:S^2 \to S^2$ are homeomorphisms.  So, if one
wants to escape the invertible world, one needs to allow critical points to be present. The best generalization is the following:

\begin{definition} We say that $f:S^2\to S^2$ is a branched covering if for every point $y\in S^2$ there exists a topological disk $V$ containing $y$ such that $f^{-1}(V)$ can be written as a disjoint union $f^{-1}(V) = \cup_{i\in I} U_i$ of open sets $U_i$ with the following property:
For every $U = U_i$ there is a unique $x \in U \cap f^{-1}(y)$ and an integer $k\geq 1$ such that the restriction $f|_{ U\backslash\{x\}} : U\backslash\{x\} \to V \backslash\{y\}$ is a covering map of
degree $k$.
\end{definition}
 
The number $k\geq 1$ is called the \textit{local degree} of $f$ at $x$, and we will denote it by
$\deg(f, x)$. The set of \textit{critical points} of $f$ is
$\operatorname{Crit}(f) = \{x \in S^2 : \deg(f, x) \geq 2\}$, and the image of a critical point is called a \textit{critical value} of $f$. The \textit{multiplicity} of a critical point $x$ is
$\deg(f, x) - 1$.

When a branched covering $f:S^2\to S^2$ has exactly two critical points $\{p,q\}$ verifying $f^{-1}(\{p,q\})=\{p,q\}$,
then $f$ restricts to a self-covering of the open annulus. Indeed, the punctured sphere $S^2\setminus\{p,q\}$ is homeomorphic to the open annulus
\[
A=S^1\times \R.
\]
Moreover, the restriction
\[
f|_{S^2\setminus\{p,q\}}:S^2\setminus\{p,q\}\to S^2\setminus\{p,q\}
\]
is a covering map. For a sphere branched covering, this specific critical configuration is very restrictive.
Nonetheless, annulus coverings (or just maps) have their own interest, and their theory  was developed in \cite{iprx1}, \cite{iprx2} and \cite{iprx3}, very much focusing on the growth
rate inequality.
We review the key points of the theory in this section.

\subsection{Semiconjugacies}

The most basic question one can ask about a self-map $f:A\to A$ is whether it factors through a circle covering.  More precisely, let $m_d:S^1\to S^1$, $m_d(z)=z^d$, and let $g_*$ be the induced map on homology groups of a given map $g$. Does there exist a continuous map 
$h:A\to S^1$ such that $h_*=\id$ and that $hf=m_dh$?

We will show in this section that any continuous map of the \textit{closed} annulus of degree $d$ ($|d|>1$) is semi-conjugate to $z^d$ on $S^1$.  Moreover, if $f:A\to A$ has degree $d$ and $K\subset A$ is compact, connected,  essential (i.e. not contained in a topological disc of $A$) and invariant,  then $f|_K:K\to K$ is semi-conjugate to $z^d$ on $S^1$.  However, this is not true for maps acting on the \textit{open} annulus $A$ (as the ones arising from restricting a branched covering of the sphere with a set of two points that is totally invariant); even if they are covering maps. We refer the reader to Section \ref{contra} for an example of $f:A\to A$ a covering map of degree $d$ that is not semi-conjugate to $z^d$ acting on $S^1$.  
We denote $\tilde A = \R\times (0,1)$ the universal cover of $A$.

\begin{theorem}[\!\!\cite{iprx1}] \label{semi}
Let $A=S^1\times (0,1)$ and  $f:A\to A$ a continuous map such that $\deg(f)=d, |d|>1$.  Let $K\subset A$ be a compact set such that $f(K)\subset K$ and let
$F:\tilde A\to \tilde A$ be a lift of $f$.  Then, 
there exists a continuous map $H:\tilde K\to \R$ such that:

\begin{enumerate}
 \item $H((x,y)+(1,0))=H(x,y)+1$ for all $(x,y)\in \tilde K$
 \item $HF=dH$
 \item there exists $M>0$ such that $|H(x,y)-x|<M$ for all $(x,y)\in \tilde K$.
\end{enumerate}
\end{theorem}

\begin{proof}
First note that the space
\[
\begin{aligned}
\mathcal{H}=\{\,H:\tilde K\to \R \mid\;& H \text{ is continuous,} \\
& H((x,y)+(1,0))=H(x,y)+1 \\
& \text{for all } (x,y)\in \tilde K \,\}
\end{aligned}
\]
is a complete metric space with the supremum metric. Moreover, if we define
\[
T(H)=\frac{HF}{d}, \quad H\in \mathcal{H},
\]
we have:

\begin{enumerate}
 \item $T:\mathcal{H}\to \mathcal{H}$,
 \item $T$ is a contraction.
\end{enumerate}

To see this, we perform the following computations:
\begin{enumerate}

 \item $T(H)\in \mathcal{H}$ because
 \begin{align*}
    T(H)((x,y)+(1,0))
    &= \frac{HF}{d}((x,y)+(1,0)) \\
    &=\frac{HF((x,y)+(1,0))}{d} \\
    &=\frac{H(F(x,y)+(d,0))}{d} \\
    &=\frac{H(F(x,y))+d}{d} \\
    &= \frac{H(F(x,y))}{d}+1 \\
    &=T(H)(x,y)+1.
\end{align*}
 
\item Let $\rho$ be the supremum metric on $\mathcal{H}$. Then
\[
\begin{aligned}
\rho(T(H_1),T(H_2))
&= \sup_{(x,y)\in \tilde K} |T(H_1)(x,y)-T(H_2)(x,y)| \\
&= \sup_{(x,y)\in \tilde K} \left|\frac{H_1F}{d}(x,y)-\frac{H_2F}{d}(x,y)\right| \\
&\le \frac{1}{|d|}\sup_{(x,y)\in \tilde K} |H_1(x,y)-H_2(x,y)| \\
&< \rho(H_1,H_2).
\end{aligned}
\]
Hence $T$ is a contraction.
\end{enumerate}
It follows that $T:\mathcal{H}\to \mathcal{H}$ has a fixed point $H$. By definition, we obtain:

\begin{enumerate}
 \item $H((x,y)+(1,0))=H(x,y)+1$ for all $(x,y)\in \tilde K$,
 
 \item $HF=dH$.
\end{enumerate}

Moreover, since $H$ is continuous and periodic, there exists $M>0$ such that
\[
|H(x,y)-x|<M \quad \text{for all } (x,y)\in \tilde K.\qedhere
\]
\end{proof}

\begin{corollary} If $f$ is a degree $d$ continuous self-map of the closed annulus and $|d| > 1$, then
$f$ is semiconjugate to $m_d$ acting on $S^1$.
\end{corollary}

\subsection{Nielsen theory}

Naturally, we want to use the semiconjugacy to obtain periodic orbits (we are interested in an exponential growth rate of periodic orbits for sphere
endomorphisms).  When dealing with annulus maps, Nielsen theory gives us exactly what we need:

\begin{theorem}[\!\!\cite{iprx2}]
Let $f:A\to A$ be a continuous map where $A=S^1\times (0,1)$, let $\deg(f)=d$ and $|d|>1$. 
If $\fix(F)\neq\emptyset$ for every lift $F:\tilde A\to \tilde A$, then $\fix(f)\geq |d-1|$.
\end{theorem}

\begin{proof}
If $F:\tilde A\to \tilde A$ is a lift of $f$, then any other lift of $f$ is of the form 
$F+k: \tilde A \to \tilde A$, given by $(F+k)(x)= F(x)+(k,0)$ for some $k\in \Z$.

We denote $F_k=F+k$ and write $F(x)+(k,0)=F(x)+k$ for all $k\in \Z$.

For each $k\in \Z$, let $x_k\in \tilde A$ be such that $F_k(x_k)=x_k$. 
We want to count how many distinct points $p(x_k)$ there are, 
where $p:\tilde A \to A$ is the covering projection.  
To do this, we study the equation $x_k=x_l+j$, with $j,k,l\in \Z$.  
Observe that $x_k=x_l+j$ if and only if

\[
\begin{aligned}
x_k &= F_k(x_k)=F_k(x_l+j)=F_k(x_l)+dj \\
    &= F(x_l)+k+dj = F_l(x_l)-l+k+dj \\
    &= x_l - l + k + dj = x_k - j - l + k + dj,
\end{aligned}
\]
which holds if and only if 
\[
k-l = j - dj = j(1-d).
\]

That is, $p(x_k)=p(x_l)$ if and only if $k-l \in (1-d)\Z$.  
Therefore, the points $p(x_0), p(x_1), \ldots, p(x_{|d-1|-1})$ are all distinct, 
and hence $\fix(f)\geq |d-1|$.
\end{proof}

We are thus led to the problem of finding fixed points for lifts $F:\tilde A\to \tilde A$.  
Observe that $\tilde A$ is topologically a plane, and moreover, if  
$f:A\to A$ is a covering map, then $F:\tilde A\to \tilde A$ is a plane \textit{homeomorphism}.

This simple observation is key: if we are dealing with coverings, we can now apply the fixed point theory for homeomorphisms of the plane, which is of course highly developed. 
In particular, in the case where the orientation is preserved, we can use Brouwer's Theory.  And even in the general case, when the map $f:A\to A$ is just continuous, we have
Lefschetz index defined on simple closed curves that will come in handy.

\begin{theorem}[\!\!\cite{iprx2}]\label{ka} 
Let $f:A\to A$, where $A=S^1\times (0,1)$, be a covering map with $\deg(f)=d$ and $|d|>1$. 
Suppose moreover that there exists a compact essential continuum $K\subset A$ such that $f(K)\subset K$.  
Then $\fix(F)\neq\emptyset$ for every lift $F:\tilde A\to \tilde A$.
\end{theorem}

Note that this result is strictly a consequence of degree; an irrational rotation of the open annulus has no periodic points, and has every essential circle as a compact invariant subset. We include here the proof in the orientation preserving case, to show how it relates beautifully to Brouwer's theory:  

\begin{proof}{\em (orientation preserving case).}
Let $F:\tilde A\to \tilde A$ be a lift of $f$ and $H:\tilde K\to \R$ be given by Theorem \ref{semi}.  
Observe that $H^{-1}(0)\subset \tilde K\subset \tilde A$ is nonempty, since $K$ is essential.  
Moreover, it is compact and invariant; therefore, since we are assuming that $f$ preserves orientation,  
Brouwer's Fixed Point Theorem ensures that $\fix(F)\neq \emptyset$.
\end{proof}

When the orientation is reversed, we use Kuperberg's theorem \cite{krys} instead.  We refer the reader to \cite{iprx2} for details.

As a consequence, we have:

\begin{theorem}  Let $f:S^2\to S^2$ be such that it restricts to an annulus covering $f:A\to A$. Suppose, moreover that there exists an essential invariant continuum $K\subset A$ such that $f(K)\subset K$. Then, $f$ has the rate.
\end{theorem}

It is worth comparing this result against the fixed-point free degree $2$ covering example $(r,\theta)\mapsto
(2r,2\theta)$, where every point is wandering.  One may ask if the existence of a
non-wandering point is enough to assure the rate (or even the existence of a fixed point).  Example \ref{e4} in Section \ref{expls} shows that this is not the case.

\subsection{Lefschetz index}\label{Lefschetz}

There is another way to get fixed points for plane maps, that does not require the homeomorphism hypothesis: the Lefschetz fixed point theorem. We recall that the Lefschetz index $i(\gamma,F)$ is defined over a simple closed curve $\gamma$ that is disjoint from the fixed point set of $F$.  It equals the degree of the circle map $t\mapsto \frac{F(\gamma(t))-\gamma(t)}{||F(\gamma(t))-\gamma(t)||}$ (we are assuming here $t\mapsto \gamma(t)$ is a homeomorphism defined on $S^1$).

The following lemma is well-known:

\begin{lemma}\label{aldo} If $F:\R^2\to \R^2$ is continuous, and $\gamma \subset \R^2$ is a simple closed curve such that $i(\gamma,F)\neq 0$, then there exists a fixed point of $f$ in the bounded component of the complement of $\gamma$.
\end{lemma}

We include Aldo Portela's
proof here, as it is the easiest,  most elegant that we know of:

\begin{proof}(Aldo Portela).  Let $U$ be the bounded 
component
of the complement of $\gamma$. If $F$ has no fixed points  in $U$, then $i(\alpha, F)=i(\gamma,F)\neq 0$, where $\alpha$ is any simple closed curve contained in $U$.  In particular, 
we can take $\alpha=\partial B(x,\epsilon)$, where $x\in U$ and $\epsilon$ is small enough such that $f(B(x,\epsilon))$ is contained in the complement of the closure of 
$B(x,\epsilon)$.  It is now straightforward to compute $i(\alpha, F)=0$, a contradiction.
\end{proof}

In what follows we relate Lemma \ref{aldo} to the rate problem. We have seen in the previous section that in order to have the rate for $f:A\to A$, we have to find fixed points for any lift $F:\tilde A\to \tilde A$. The universal cover $\tilde A$ can be thought as having four ``ends'': bottom, upper, right and left. We will refer to proper lines connecting the bottom and upper ends of $\tilde A$ as ``vertical'' lines, and to
proper lines connecting the left and right ends of $\tilde A$ as ``horizontal'' lines.

Another simple idea gives the 
rate for maps of the annulus of degree $d$, $|d|>1$: any lift $F$ will push vertical lines that are sufficiently far to the right or left strictly to the right of themselves, or strictly to the left of 
themselves (we are using $|d|>1$ here).  So, with hypothesis guaranteeing that horizontal lines are also pushed strictly above them, or strictly below them, one can calculate the Lefschetz index of a simple closed curve formed by intersecting appropriate pairs of vertical and horizontal lines.  This is formalized as follows:

\begin{lemma}[\!\!\cite{iprx3}] \label{indice}
Let $\alpha$ and $\beta$ be disjoint simple proper lines in the plane, each one of which separates the plane. Let $\gamma$ and $\delta$ be another pair of disjoint curves separating
the plane. Assume also that each $\gamma$ and $\delta$ intersect $\alpha$ in one point and $\beta$ in one point. Now let $\Gamma$ be the simple closed curve determined by the four
arcs of the curves delimited by the intersection points, with the positive orientation.
Now let $f$ be a map defined on $\Gamma$ such that $f(\Gamma\cap\alpha)$ is contained in the component of the complement of $\alpha$ that contains $\beta$,
$f(\Gamma\cap\beta)$ is contained in the component of the complement of $\beta$ that contains
$\alpha$, that $f(\Gamma\cap\delta)$ is contained in the component of the complement of $\gamma$ that does not contain $\delta$ and that
$f(\Gamma\cap\gamma)$ is contained in the component of the complement of $\delta$ that does not contain $\gamma$. Then the index of $f$ in $\Gamma$ is equal to $1$.
\footnote{Many years later I realized that Lemma \ref{indice} can be used to give a proof of Brouwer's theorem for orientation preserving plane homeomorphisms.  This proof
is included in Section \ref{newbrou}.}
\end{lemma}

Moreover, a similar statement can be proved when $f(\Gamma\cap\delta)$ is contained in the component of the complement of $\delta$ that does not contain $\gamma$ and
$f(\Gamma\cap\gamma)$ is contained in the component of the complement of $\gamma$ that does not contain
$\delta$; in this case $I_f(\Gamma)= -1$. Therefore, we have the following:

\begin{theorem}[\!\!\cite{iprx3}] \label{t3}
Let $f: A \to A$ be a degree $d$ map of the annulus, where $|d|>1$.  Each one of the following conditions imply that $f$
has the rate.
\begin{enumerate}
\item
Both ends of $A$ are attracting.
\item
Both ends of $A$ are repelling.
\end{enumerate}
\end{theorem}

\begin{theorem}[\!\!\cite{iprx3}]\label{sarkoski} Let $F$ be a map of the open annulus $A$ that interchanges the ends of $A$. If $|\deg(F)|>1$, then $F$ has the rate.
\end{theorem}

Note that this last one can be seen as a forcing, Sarkovskii-type, result: let $f:S^2\to S^2$ be a degree $d$ map, where $|d|>1$.  If $f$ has a $2$-periodic totally invariant cycle, 
then it has the rate.

\section{Thurston maps with parabolic orbifolds}\label{para}

We saw in the previous section that if a  branched covering $f:S^2\to S^2$ has exactly two critical points $\{p,q\}$ verifying $f^{-1}(\{p,q\})=\{p,q\}$,
then $f$ restricts to a self-covering of the open annulus.  For a sphere branched covering, this specific critical configuration is very restrictive, but as we have seen, quite
interesting.

In this section we dive further in this direction and focus on the structure of the postcritical set: are there hypotheses on this set that guarantee the desired growth rate? We 
showed in \cite{thurston} that
for Thurston maps with parabolic orbifolds either the growth rate is satisfied, or the structure of the postcritical set is exactly as in the counterexample 
$(r,\theta)\mapsto (2r, 2\theta)$:
there are exactly two critical points which are fixed and totally invariant.  We will define Thurston maps and parabolic orbifolds soon, but we first give a little context. 

Thurston maps with parabolic (also called non-hyperbolic) orbifolds were completely classified in \cite{dh}. It turns out that these maps can be lifted to degree $d$ covering maps
$(|d|>1)$ of either the open annulus or the torus, or even act directly on the open annulus by restriction (for example, when
there are exactly two totally invariant critical points as in the map $(r,\theta)\mapsto (2r, 2\theta)$). So topologically, parabolic orbifolds are quotients of torus or open annulus 
endomorphisms\footnote{The converse statement is also true: quotients of torus and annulus endomorphisms have parabolic orbifolds. This was proved by the author and S. Llavayol in 
\cite{qote}.}, where the quotient may be trivial. 

The theory of torus endomorphisms is of course well understood,
even from the point of view of the growth rate inequality, because of compactness and homological considerations.  The theory of open annulus endomorphisms was discussed in 
Section \ref{ac} of the present work. As a consequence, these very special maps (Thurston maps with parabolic orbifolds) provide a perfect environment to tackle 
this elusive growth rate problem.

The growth rate inequality for \textit{expanding} Thurston maps was already known.  It was shown in \cite{Li} that each \textit{expanding} Thurston map has $\deg f +1$ fixed points, counted
with appropriate weight.  The definition of an expanding Thurston map is rather technical and we will not discuss it here, but we refer the interested reader to the excellent 
book \cite{bm}.

We do not assume our maps
to be expanding, but
more importantly, even when they are expanding, our proofs are entirely different.  We gave purely topological proofs of the results, using only elementary algebraic topology and 
the Lefschetz fixed point theorem.

\subsubsection*{Definitions and notations}  If $f:S^2\to S^2$ is a branched covering, we denote by $\deg_xf$ the local 
degree of $f$ at $x$.  We will call $S_f=\{x\mid \deg_x f>1\}$ the critical set of $f$, and $P_f=\cup_{n>0}f^n(S_f)$ the postcritical set.

A \textit{Thurston map} is an orientation preserving branched covering of the 
sphere onto itself such that the postcritical set $P_f$ is finite. The \textit{ramification function} $\nu_f$  of a Thurston map $f$ is 
the smallest among functions $\nu:S^2\to \N^*\cup\{\infty\}$ such that:
\begin{itemize}
 \item $\nu(x)=1$ if $x\notin P_f$;
 \item $\nu(x)$ is a multiple of $\nu(y)\deg_y(f)$ for each $y\in f^{-1}(x)$.
\end{itemize}
                                                                        
The \textit{orbifold} associated to a Thurston map $f$ is the pair $\mathcal{O}_f=(S^2, \nu_f)$. Note that $\{p\in S^2:\nu_f(p) \geq 2\}=P_f$ is a finite set.  If we label these points
$p_1,\dots,p_n$ such that $2\leq\nu_f(p_1)\leq\ldots \leq \nu_f(p_n)$, then the $n$-tuple $(\nu_f(p_1), \ldots, \nu_f(p_n))$ is called the \textit{signature} of $(S^2, \nu_f)$.
Thus the numbers appearing in the $n$-tuples are the 
values of the ramification
function, which assign ``weights'' to points in $P_f$.

The \textit{Euler characteristic} of $(S^2, \nu_f)$ is 
\[\chi (\mathcal O_f)=2-\sum_{x\in P_f} \left(1-\frac{1}{\nu_f(x)}\right).\] 

A Thurston map has \textit{parabolic} orbifold if $\chi (\mathcal O_f)=0$. A Thurston map has a parabolic orbifold if and only if $\deg_p f \cdot \nu_f(p) = \nu_f(f(p))$ for all $p\in S^2$
and if and only if the signature of $\Of$ is   $(\infty, \infty)$, $(2,2,\infty)$, $(2,4,4)$, $(2,3,6)$, $(3,3,3)$ or
$(2,2,2,2)$ (see, for example, \cite[Proposition 2.14]{bm}).  If there is no place for confusion,  we will often write $\nu$ instead of $\nu_f$ for the ramification function.

\begin{theorem}[\!\!\cite{thurston}]\label{thurstonthm} Let $f: S^2 \to S^2$ be a Thurston map with parabolic orbifold and degree $d$, $|d|>1$. Then either the growth rate inequality 
$\displaystyle \limsup_{n\to\infty} \frac{1}{n} \log (\# \fix (f^n))\geq \log |d|$ holds for $f$ or 
$f$ has exactly two critical points which are fixed and totally invariant.
\end{theorem}

It follows that a Thurston map with parabolic orbifold either satisfies the growth rate inequality or  the signature of $f$ is $(\infty, \infty)$ and
$P_f=S_f=\{p,q\}\subset \fix(f)$.

The proof of this theorem is done case by case, analyzing the six possible signatures for parabolic orbifolds.  
Regarding the signature $(\infty, \infty)$, we point out the following: if a map
in this class does not  satisfy the growth rate, then necessarily $p$ and $q$ are fixed (and not a period-2 critical cycle). 
This is stated in \cite[Remark 4(2)]{iprx3} and was explained in Theorem \ref{sarkoski} in the previous section.

\begin{proof}({\em Sketch in the general case:})
The key is to show that each parabolic map
\[
f:S^2\to S^2
\]
lifts to a map
\[
F:X\to X
\]
under a branched covering projection
\[
\pi:X\to S^2,
\]
where $X$ is either the annulus or the torus, and $\pi$ is finite-fold.

Therefore, if a lift $F:X\to X$ has the rate, then $f$ also has the rate. Indeed, for every iterate $k$, the fixed points of $f^k$ are at least a fixed fraction of those of $F^k$.

In order to prove the existence of the lifts, we use the classical lifting criterion: parabolicity provides exactly the homotopical condition required for lifting.

In order to prove that a lift $F:X\to X$ has the rate one uses the deck transformations of the covering projection $\pi$.  If one lift $F$ does not have the rate, then there exists
$T \in \Deck(\pi)$ such that $TF$ has the rate.  In the annulus case, which reduces to the $(2,2,\infty)$ signature,  one uses again Theorem \ref{sarkoski}, as if one lift $F$ fixes the end of
the annulus, then the other one interchanges them ($\pi$ is the usual two-fold branched covering from the annulus to the plane in this case). For each of the remaining (toral)
signatures, one just needs to check that the lift $F:\T^2\to \T^2$ has hyperbolic action on homology in order to verify the rate. Indeed, it follows from \cite[Theorem 1.2, p.~618]{bfgj} that if all eigenvalues of $F_*:H_1(\T,\Z)\to H_1(\T,\Z)$ are different from $1$, then
$F$ satisfies the growth rate inequality.  Again, if one lift $F:\T^2\to \T^2$ has non-hyperbolic homology action, one can find $T \in \Deck(\pi)$ such that $TF$ acts hyperbolically
on homology.  This is due to the very specific geometries of the branched coverings $\pi$ arising from parabolic orbifolds.  For much more on this last point we refer the reader 
again to \cite{bm}.
\end{proof}

To finish this section, we point out the following. We have seen that parabolic orbifolds come with a surface covering associated to them, and therefore with a plane homeomorphism
associated to them.  As already explained, pairing this together with Brouwer and Nielsen theory guarantees the growth rate inequality in this setting, apart from a well
known exception.  This means that for parabolic orbifolds, the obstructions for not verifying the growth rate inequality are completely understood: $f$ has exactly two critical points 
which are fixed and totally invariant. It would be then desirable to have a program to link the theory of plane homeomorphisms with hyperbolic orbifolds as well.

\section{Topological polynomials}\label{polis}
In this section, we still consider sphere branched coverings, but drop the parabolicity condition and  make the following assumption instead: there exists a proper, open, simply-connected and completely invariant region 
$R\subset S^2$.
These two assumptions (branched covering + completely invariant simply connected region) are strong assumptions,
but there will be more, since many examples of maps not having the rate satisfy these two assumptions (we discuss many of these examples in Section \ref{expls}). These examples also
satisfy that there are exactly two fixed critical points of
multiplicity $d-1$, one in $R$ and the other one in the boundary of $R$.   It follows that $f$ can be thought of as a covering map of the open annulus $\R ^2 \backslash \{0\}$, and such
maps were discussed in Section \ref{ac}.

In the case that the boundary 
of $R$ is locally connected, we showed that outside this very particular post-critical configuration, the growth rate inequality is verified:

\begin{theorem}[\!\!\cite{bcs}]\label{teoa} Let $f$ be a degree $d$ branched covering of the sphere, where $|d|>1$. Assume that there exists a completely invariant simply connected region $R$ whose boundary is locally
connected. Assume moreover that it is not the case that there exists only one critical point in the boundary of $R$ that has multiplicity $d-1$ and is fixed by $f$. Then, $f$ has
the rate.
\end{theorem}

We do not know if the hypothesis on local connectivity is necessary.
A main ingredient in the proof of Theorem \ref{teoa} is that  $f$ extends continuously to the prime end closure of $R$. This extension of $f$ induces a map $\tilde f$ of the circle that turns 
out to be a 
degree $d$ covering, despite the fact that $f$ could have critical points on the boundary of the region $R$. It follows that for every positive $n$, $\tilde f^n$ has at least
$|d^n-1|$ fixed points. As the boundary of $R$ is locally connected, to each periodic point of $\tilde f$ corresponds a periodic point of $f$ in the boundary of $R$. However, this 
correspondence is not injective, so in order to get the rate one has to understand how many different $\tilde f$-periodic prime ends correspond to the same point in the boundary of $R$.

An example to have in mind is when $f$ is a complex polynomial with connected and locally connected Julia set. Then $f$ has a superattracting fixed point at infinity and the region $R$ is
its basin of attraction, which is the complement of the filled Julia set. Figure \ref{rayos} shows different periodic rays landing at the same point, namely the periodic orbit
$1/7 \to 2/7 \to 4/7$ is reduced to a point in the boundary of $R$.  More figures illustrating this
phenomenon can be found in \cite[Chapter 18]{milnor}.
\begin{figure}[ht]
    \centering
    \psfrag{17}{$\frac{1}{7}$}
    \psfrag{27}{$\frac{2}{7}$}
    \psfrag{47}{$\frac{4}{7}$}
    \includegraphics[scale=0.3]{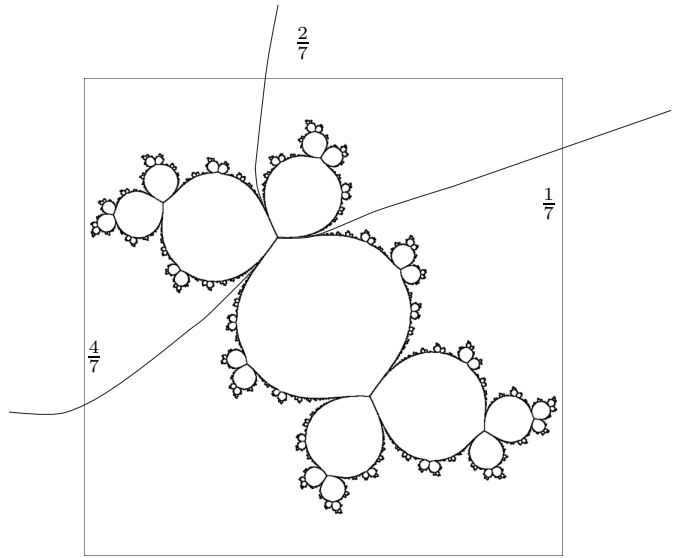}
    \caption{Julia set for $f(z) = z^2 - 0.110 + 0.6557i$}
    \label{rayos}
\end{figure}

So, in particular,
Theorem \ref{teoa} is a topological version of the fact that complex polynomials (with connected and locally connected Julia set) have the rate.  Following the polynomial analogy, the opposite 
situation
corresponds to the case when all critical points belong to the complement of the filled Julia set. We also gave a topological version of the fact
that such maps have the rate:

\begin{theorem}[\!\!\cite{bcs}]\label{teob} Let $f$ be a degree $d$ branched covering of the sphere, where $|d|>1$. Assume that there exists  a simply connected open set $U$ whose closure is disjoint from the set of
critical
values and
such that $\overline{f^{-1} (U)}\subset U$. Then $f$ has the rate.
\end{theorem}

Going back to Theorem \ref{teoa}, we already pointed out that local connectivity is a technical hypothesis used only to guarantee that external rays land.  If one wanted to push the proof 
forward, one has to consider the case where external rays land on proper subcontinua of $\partial R$ and try to fish for fixed points inside these continua. This strategy could work
if for instance one could prove that the impression of periodic rays (or at least of enough of them) is contained in an injectivity region of $f$.  Indeed, if that were the case, 
one could get fixed points of the iterates $f^n$ by applying some Brouwer theory locally.  However, it may as well
happen that the impression of every (periodic) prime end is the whole set $\partial R$, in which case this particular strategy falls apart.  The question that arises then is 
whether or not
that is possible.

It is fascinating that this is not known even in the holomorphic case.  Indeed, whether or not there exists a rational function with an
indecomposable continuum as its Julia set is a well-known unsolved
problem (see \cite{cmr}).  Furthermore, it was not known until recently if there existed \textit{branched coverings} of the sphere of degree $d$, with $|d|>1$, having an indecomposable 
completely invariant continuum. In \cite{plig} such an example was given, and we will discuss it in the next section. However, this example does not have a completely invariant
proper and simply-connected region.  Moreover, it is not known whether or not there exists a \textit{plane}
branched covering of degree $d$, with $|d|>1$, 
having an indecomposable 
completely invariant continuum.

In \cite{elena} the results in \cite{cmr} were studied from a topological viewpoint, and were shown to be purely topological (independent of the 
holomorphic assumptions).

It is known that \textit{certain} indecomposable continua cannot be totally invariant subsets of a sphere branched covering, and this will be discussed in
Section \ref{indec}.  However, intriguingly enough, the general question  stands wide open.

\section{Indecomposable continua}\label{indec}

As explained in the previous section, the hypothesis of local connectivity does not seem to be relevant in Theorem \ref{teoa}
and it is plausible that this hypothesis should be replaced by decomposability instead.  It is then natural to try to understand the examples presenting totally invariant
indecomposable continua.  This turns out to be a rather intricate endeavour.

We first point out that it is well known that surface homeomorphisms admit indecomposable continua as invariant subsets.  There are several methods to construct these
examples, and they can even be made minimal; in particular, periodic-point-free (see, for example \cite{handel} and \cite{bama}).  
It is also well-known that the solenoid admits non-trivial connected coverings (see \cite{sole}), but the solenoid is not embeddable in the sphere. It is also known that for any integer $k$, the
pseudo-circle is a $k$-fold
covering space of itself (see \cite{gam} and \cite{heath}). However, the following remains unknown: does there exist a plane branched covering with a totally invariant 
pseudo-circle? That is, it is not known how to extend a self-covering of the pseudo-circle to the whole plane so that the extension is a branched covering.  It is known
how to extend it to a \textit{map} of the annulus, which was already noted in \cite{boronski}.  Moreover, in \cite{boronski} the rate was proven to hold 
for annulus maps preserving a pseudo-circle and this result was generalized to all plane separating circle-like continua.

\subsection{A sphere branched covering supporting a completely invariant indecomposable continuum}  We proved the following:

\begin{theorem}[\!\!\cite{plig}]
\label{t1}
There exists a branched covering $f:S^2\to S^2$ supporting a completely invariant continuum $K$ and satisfying the following
properties:

\begin{enumerate}
\item
$f$ is smooth.
\item
$f$ has two attracting fixed points, each basin contains one of the two critical points of $f$. The restriction of $f$ to the immediate basins is injective.
\item
The compact set $K$ is the complementary set of the union of the basins: $K$ is a repellor with local product structure.
\item
$K$ is indecomposable, with infinitely many Wada lakes.
\item
$f$ satisfies the growth rate inequality.
\item
Every $C^1$ perturbation of $f$ satisfies properties $2$, $4$, and $5$ above. Every branched covering which is also a small $C^\infty$ perturbation of $f$ satisfies all 
the properties above.
\end{enumerate}
\end{theorem}

In retrospect, the idea for the construction is very simple.  We begin with a nonexpanding Anosov endomorphism of the torus with two fixed points and perform the derived from 
Anosov perturbation 
to transform both hyperbolic fixed points into sinks.  The derived from Anosov endomorphism was first introduced by F. Przytycki, see \cite[Section~6]{prz}.  In his paper,
he offered the following remark, which I quote verbatim here due to its clear relevance:

``(Caution: One usually makes a source in
the
case of a diffeomorphism, but this way would be wrong in the case
of an endomorphism).'' 
Now, one considers the set $K$ defined as the complement of the basin of attraction of the sinks we created.  In this way, we get that $K$ is a totally invariant repeller and an indecomposable continuum.  To finish our construction we just factorize to the sphere under a four points ramification branched covering, where the ramification points of the covering projection are fitted exactly to the sinks we created. Then some 
arrangements are made to obtain smoothness and the remaining properties.

There is another possible approach: the linear Anosov can be first carried to the sphere under the branched covering and then one can imitate the construction of 
the Plykin
attractor (that in our case would be a repeller).  In this case,  a derived from pseudo-Anosov perturbation must be made to obtain the  example (see \cite{bam} for the derived from pseudo-Anosov perturbation).

We explain the details of our construction in what follows. We consider the linear hyperbolic toral endomorphism $\overline A$ induced by the matrix 

$$A=
  \left[ {\begin{array}{cc}
   4 & 1 \\
   2 & 1 \\
  \end{array} } \right]
$$

The map $\overline A: \T^2\to \T^2$  is a nonexpanding Anosov endomorphism of degree $2$; its dynamics is well known, it has invariant stable and unstable foliations with dense leaves.
It is {\em special} meaning that the unstable space of a point does not depend on the preorbit chosen to define it; it depends only on the point. This is well known not to be structurally stable: generic $C^1$ perturbations do not satisfy this property (see \cite{prz}).

Consider the equivalence relation $x\sim -x$ on
\[
\T^2=\R^2/\Z^2.
\]
Note that $\T^2/\sim$ is topologically a sphere, and that the quotient map
\[
\pi:\T^2\to S^2
\]
is a degree $2$ branched covering with exactly four critical points. These correspond to the classes of the points
\[
\left(\tfrac12,\tfrac12\right),\ (1,0),\ \left(\tfrac12,0\right),\ \left(0,\tfrac12\right).
\]

Note also that $\overline A$ induces a map on the sphere simply because it is linear, but we will previously proceed to perform a derived from an Anosov endomorphism
(as introduced in \cite{prz}). 
The fixed points of $\overline A$ are $(0,0)$ and $(1/2,1/2)$. Both can be turned into attractors, creating as usual two new saddles. This procedure is well known and taking extra care to note that the
modification can be made symmetric with respect to the strong stable manifold of the fixed points. It then follows that the new map $A'$ still preserves the equivalence relation
$\sim$ in $\T^2$. As is also well known, the unstable foliation remains the same.  It follows that $A'$ induces a map on the sphere $f'$ so that $\pi A'=f'\pi$.

By construction $f'$ is a degree $2$ branched covering map. It has two critical points that can be easily found by computing those points having just one preimage. The critical points are $\pi(1/4,0)$ and $\pi(1/4,1/2)$ with corresponding critical values $\pi(0,1/2)$ and $ \pi(1/2,0)$.
Moreover, the images of the critical values are the fixed points of $f'$: $f'(\pi(0,1/2)) = \pi(1/2,1/2)$
and $f'(\pi(1/2,0)) = \pi(0,0)$.

It follows by construction that $f'$ is a $(2,2,2,2)$ orbifold (see Section \ref{para} for the definition).

Note that as $A$ is smooth, $f'$ is also smooth except exactly at the critical points and critical values where it is not even differentiable. These points belong to the 
basin of attraction of the fixed points. We show next how to make a critical point $c$ smooth, which in our linear case  can be  easily done. Note that the preimage of an adequate ellipse $E$
around $f'(c)$ is a circle $S_\rho(c)$ of radius $\rho$ centered at $c$ such that $f'$ is $2:1$  from $S_\rho(c)$ to $E$.
Moreover, lines through $c$ are carried to lines through $f(c)$, which implies that, in polar coordinates (centered at $c$ and $f'(c)$),
one has $f'(r,\theta)=(g(r,\theta),h(\theta))$.
In addition, note that by linearity again, $\frac{g(r,\theta)}{r}$ is bounded and bounded away from $0$.
Let now $\Phi$ be a $C^\infty$ function that is equal to the identity close to $S_\rho(c)$ and is equal to $(r,\theta)\to (e^{-1/r^2},\theta)$ close to $0$.  Then $f'\circ\Phi$ is
$C^\infty$ at $c$ and coincides with $f'$ outside a neighborhood of $c$.

Finally we can make a new modification of this map so that the obtained map is also smooth at the critical values (much easier because $f'$ is locally a homeomorphism at these points).
We denote by $f$ this last map. 

Note that $f$ has the rate; this is obvious from the construction, but it also follows from Theorem \ref{thurstonthm}.  We refer the reader to \cite{plig}  to see that $f$ satisfies all the remaining requirements of
Theorem \ref{t1}.

So, we have an example of a sphere branched covering supporting a completely invariant indecomposable continuum, but the example is uninteresting from the point of view of
the growth rate inequality.

Moreover --- and unfortunately --- this map is not Thurston equivalent to a rational map.\footnote{We will not define Thurston equivalence here, but we refer the reader to \cite[Section 2.4]{bm}} It is a parabolic $(2,2,2,2)$ orbifold with associated 
torus map in the homotopy class of Anosov, which is a well-known Thurston obstruction (see Section \ref{para} for the definitions and further references). So, even in the $C^0$ 
setting it is unknown if a map supports an indecomposable ``Julia set''.  More precisely, whether or not there exists a Thurston map equivalent to a rational map having an indecomposable completely invariant continuum remains unknown. And adding to the shortcomings of our example, it is not a \textit{topological polynomial}, that is, it
is not a \textit{plane} branched covering.  Plane branched coverings were treated in Section \ref{polis}, and were our main motivation to dive into dynamics of indecomposable continua.

We remark that it remains unknown whether or not there exists a \textit{plane} branched covering of 
degree $d$,
with $|d|>1$, 
having a  
completely invariant indecomposable continuum.

\subsection{The Riemann-Hurwitz formula}

At this point, the reader should be convinced that understanding the dynamics of coverings and branched coverings on indecomposable continua is essential for deepening our understanding
of the
dynamics of branched coverings $f:S^2\to S^2$.

Branched coverings of the 
pseudo-arc were first constructed in \cite{vernon}, and constructed again and used in \cite{jernej} to prove the Barge Entropy Conjecture. These are degree $d$ branched coverings with one critical point of multiplicity $d-1$. In \cite{covth} we provided
 an example of a branched covering of the Knaster continuum with a different ramification portrait, whose peculiarity is that it does not satisfy the classical Riemann-Hurwitz formula.

Recall that if $f:T\to S$ is a branched covering between compact surfaces, then the Riemann-Hurwitz formula
\[
\chi(T) = d \chi(S) - r
\] 
holds, where $\chi$ is the Euler characteristic, $d$ is the degree of $f$, and $r$ is the 
number of critical points counted with multiplicity.

 \begin{theorem}[\!\!\cite{covth}]\label{ext}  Let $f:S^2\to S^2$ be a branched covering of degree $d\geq 1$. Let $X\subset S^2$ be a non-separating continuum such that $f^{-1}(X)=X$.  Then $X$ satisfies the
 Riemann-Hurwitz formula $1=d-r$, where $r$ is the number of critical points in $X$ counted with multiplicity.
 \end{theorem}
 
 \begin{proof}  Let $\gamma$ be a simple closed curve separating $X$ and all critical values of $f$ outside $X$. Let $U$  be the connected component of $S^2\backslash \gamma$
 containing $X$.  Note that there is only one connected component of the complement of $f^{-1}(\gamma)$ containing $X$, and therefore $f^{-1}(U)$ is connected and equals this component $V$.  Now, the 
 Riemann-Hurwitz formula $1=d-r'$ holds for $f|_{V}:V\to U$, where $r'$ is the number of critical points of $f$ in $V$ counted with multiplicity.  By the choice of $\gamma$, the 
 critical points of $f$ in $V$ equal the critical points of $f$ in $X$, and the result follows.
 \end{proof}

 \begin{remark}  The previous theorem allows us to prove that certain branched coverings acting on planar continua do not extend to branched coverings of the plane or the sphere just by
 looking at the intrinsic dynamics on the continua.
 \end{remark}

 \begin{corollary}  Let $f:X\to X$ be a branched covering of degree $d$, where $X$ is a non-separating planar continuum. If $f$ extends to a degree $d$ branched covering of the sphere, then
 $f:X\to X$ satisfies the
 Riemann-Hurwitz formula $1=d-r$, where $r$ is the number of critical points in $X$ counted with multiplicity.
 \end{corollary}

One could expect that this formula would hold for branched coverings defined on arbitrary continua, where one would need to accommodate every definition (note that 
even the definition of a covering map is shady when the underlying space is not locally connected).  However, as will be explained below, even in simple cases where the definitions
adapt trivially, the formula may fail.

We  proved in  \cite{covth}  that \textit{any} branched
self-covering
of the Knaster continuum does not verify the Riemann-Hurwitz formula, and therefore cannot be extended to a branched covering of the sphere.  This result shows that 
studying
intrinsic dynamics of coverings and branched coverings on indecomposable continua is a useful tool, even if our goal is understanding their interaction with an ambient space dynamics.
It also rules out the possibility of the Knaster continuum as a totally invariant set for a branched covering of the sphere. We expect this kind of result to generalize for a 
wider class of indecomposable continua.

\begin{lemma}[\!\!\cite{covth}] \label{bck} Let $K$ be the Knaster continuum and $f:K\to K$ be a degree $d$ branched covering, $d\geq 1$.  Then $f$ has exactly one critical value $p$: the endpoint of $K$. Moreover, $p$ is fixed and 
regular, 
and any other preimage of $p$ is a critical point of multiplicity $1$.  In particular, $d$ is odd and  there are exactly $\frac{d-1}{2}$ critical points.
 \end{lemma}
 
 \begin{proof}  Note that any point other than $p$ in $K$ is locally homeomorphic to $C\times I$, where $C$ is the Cantor set and $I=(0,1)$.  This implies that $p$ must be fixed and 
 regular.  Any other preimage of $p$ is a point with local structure $C\times I$ which is mapped to the endpoint of $K$.  Then, necessarily the point is critical and its multiplicity
 is $1$ (as composants are sent to composants, the critical points are critical points of interval maps).
 \end{proof}
 
 \begin{corollary}  Let $f:K\to K$ be a degree $d$ branched covering, $d\geq 1$.  Then $f$ does not satisfy the Riemann-Hurwitz formula $d-1=r$.
 \end{corollary}

Because of Theorem \ref{ext}, this allows us to rule out the possibility of the Knaster continuum as a ``Julia set'' for a  branched covering of the sphere.

We finish this section with a question.  We have seen that indecomposability is an obstruction to proving the classical Riemann-Hurwitz formula for arbitrary (non-separating) continua.
Is it the only one?

\section{Examples}\label{expls}

This section contains a potpourri of examples that are relevant to everything discussed in this paper.

\subsection{Recurrence and periodic orbits}\label{e4}

As in the fixed-point-free degree-$2$ covering example $(r,\theta)\mapsto
(2r,2\theta)$, every point is wandering, so one may ask if the existence of a
non-wandering point is enough to assure the existence of a fixed point. The
next example shows that this is not the case.

We will construct a degree $2$ covering $f:(0,+\infty)\times S^{1}\to(0,+\infty)\times S^{1}$ such that there is a compact set $K$ satisfying $f(K)= K$ and $\per(f) = \emptyset$. 
In this example $K$ is a Cantor set (recall that if $K$ is an essential continuum, then $f$ has the rate by Theorem~\ref{ka}).

We first explain how to construct a degree $2$ circle covering having a wandering interval (we call this map a ``degree $2$ Denjoy'' and this construction was communicated to us
by Edson de Faria).

Let $g_1: S^{1}\to S^{1}$ be a Denjoy homeomorphism with a wandering interval $I$. Take an open interval  $I_0\subsetneq I$ and an increasing function 
$h:I\to S^{1}$ such that $h(I_0)=S^1$ and $h|_{I\setminus I_{0}}\equiv g_1$ (see Figure \ref{figura1} (a)). Let $g: S^{1}\to S^{1}$ be the map
\begin{equation*}
g(x)= \left\{
\begin{array}{l}
g_{1}(x) \mbox{ if $x\notin I$} \\
h(x) \ \mbox{ if $x\in I$}\\ 
\end{array}
\right.\\
\end{equation*}

Note that $g$ is a degree $2$ covering of the circle and $g_1(I)$ is a wandering interval for $g$. Besides, if $x_0\in  g_1(I)$ then $K_1 = \omega _g(x_0)$ is a Cantor set and
$K_1\cap \per(g) = \emptyset$.

We are now ready to construct our example  $f:(0,+\infty)\times S^{1}\to(0,+\infty)\times S^{1}$, which has the form $f(r,\theta )=(\phi (r,\theta),g(\theta ))$, where $\phi$ is to be
constructed. Let $\varphi:(0,+\infty )\to (0,+\infty)$  be as in Figure \ref{figura1} (b).  Define  $\phi  (r,\theta)=\varphi(r)+r d (\theta,K_1)$.

\begin{figure}[ht]
    \centering
    \psfrag{0}{$0$}
    \psfrag{y=x}{$y=x$}
    \psfrag{1}{$1$}
    \psfrag{g}{$g$}
    \psfrag{f}{$\varphi$}
    \psfrag{11}{$\frac{1}{2}$}
    \psfrag{i}{$I$}
    \psfrag{h}{$h$}
    \psfrag{g1}{$g_1$}
    \subfigure[]{\includegraphics[scale=0.2]{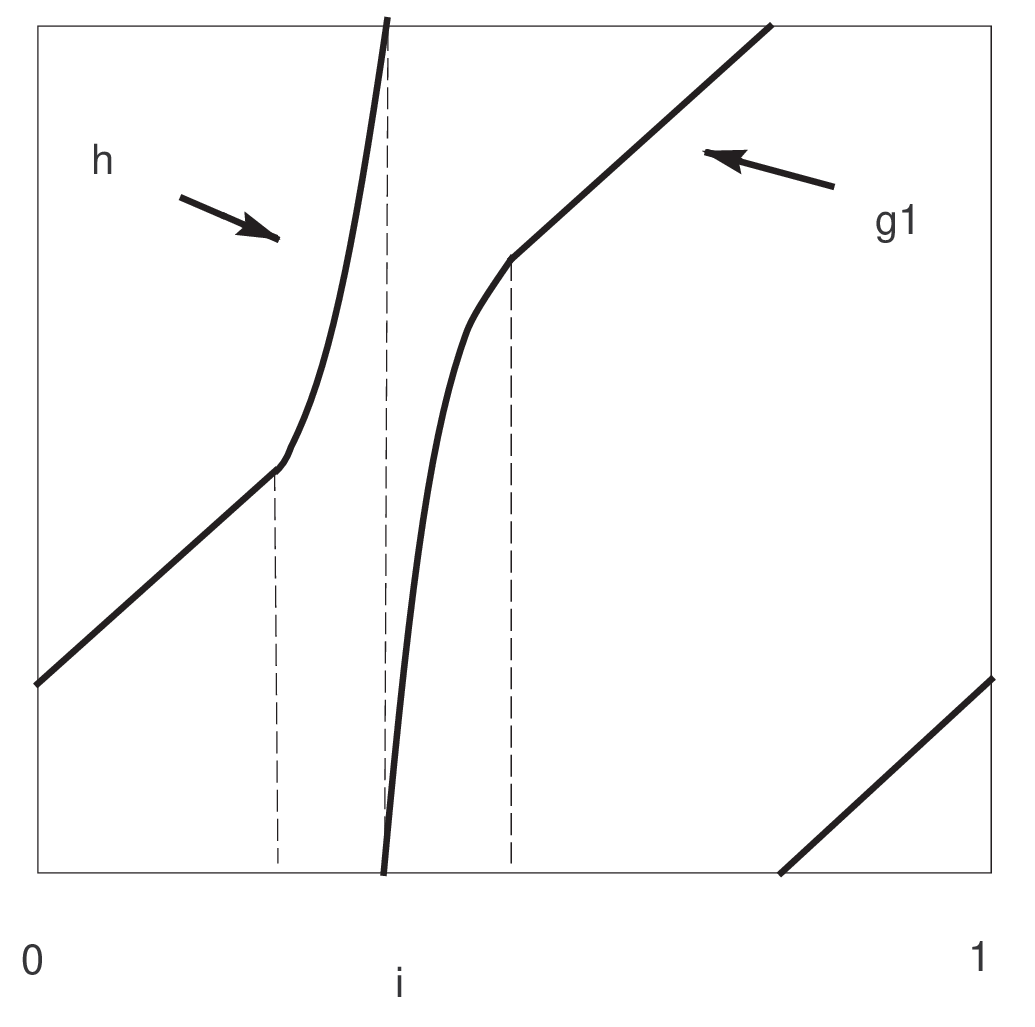}}
    \subfigure[]{\includegraphics[scale=0.2]{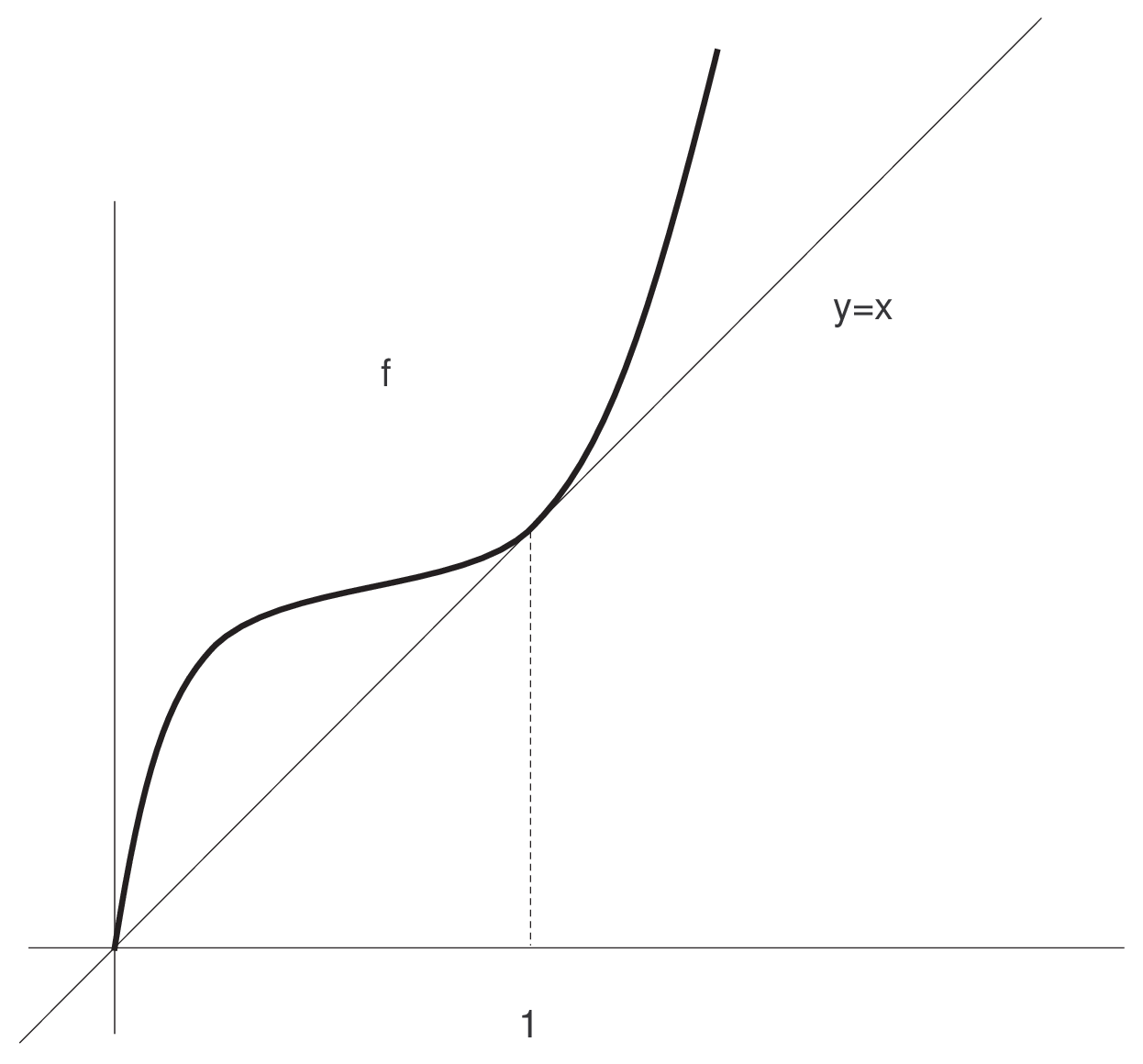}}
    \caption{}
    \label{figura1}
\end{figure}

Note that $f$ has the following properties: 
\begin{enumerate}
\item $K=\{1\}\times K_1$ is compact and  $f(K)=K$.
\item $f$ has no periodic points in $K$, because $K_1\cap \per(g) = \emptyset$.
\item If $\theta$ is periodic for $g$, then $f^n (r,\theta)\to \infty$.
\end{enumerate}
Therefore,  $\per(f) = \emptyset$.

Note also that this example can be made $C^1$, if we use the square of the distance in the definition of $\phi$ and enough regularity for the rest of the functions.

\subsection{A Kan-like example without periodic points}

We have given two examples of covering maps of the annulus without periodic points.
The first one, given in the introduction, is $(r,\theta)\mapsto
(dr,d\theta)$.  Every point is wandering in this example, wandering from one end of the annulus to the other. The second one was given above in Example \ref{e4}: in
this case the map has nonempty nonwandering set.  Nevertheless, every wandering point wanders from one end of the annulus to the other.  Note also that if
both ends are attracting or both repelling, or if the ends of the annulus are exchanged, then the map has the rate (see Theorems \ref{t3} and \ref{sarkoski} in 
Section \ref{Lefschetz}). In
all the examples we have seen so far, if a map does not have the rate, one end is attracting and the other
repelling. Is this necessary?

Let $f : S^1\times (0, 1)\to  S^1\times (0, 1)$, $ (z, x) \mapsto (z^3 , \varphi_z (x))$, where $\varphi_z$ is an increasing
homeomorphism of the interval $(0, 1)$ for each $z$, such that $\varphi_1^n (x) \to 1 $ and
$\varphi_{-1}^n (x)\to 0 $ for every $x$. This implies that the ends of the annulus are neither attracting nor
repelling and that the map has no fixed points. We will show
here that such a map can be constructed without any periodic
point. This example was communicated to us by the referee of \cite{iprx2}.

Let $A = S^1\times \R$ and assume that the map is given by
$f(z, x) = (z^3 , x + t_z )$
where $t_z$ varies continuously with $z$ and the set of numbers $\{t_p\}$ with $p$
periodic of $m_3$ is rationally independent. This means that $f$ has no periodic
points, because
if $\{p_1, \ldots, p_n\}$ is a 
periodic orbit of $z^3$, then $f^n(p_1, x) =(p_1, x +\sum_i t_{p_i})$, but by assumption $\sum_i t_{p_i}\neq 0$.
To construct the function $t_z$
choose a rationally independent sequence $\{a_i : i \geq 0\}$ of positive numbers
and enumerate the periodic points of $m_3$ as $\{p_n : n \geq 0\}$, with $p_0 = 1,
p_1 = -1$. Define by induction a sequence of functions $z\mapsto t^n_z$ beginning with
any continuous function $t^0_z$ such that $t^0_{p_0} = a_0$ and $t^0_{p_1} = 0$. Given $n > 0$
define $t^n_z$ as follows: $t^n_z = t^{n-1}_z$
outside a neighbourhood of $p_n$ not containing any $p_i$ for $i < n$, $ 0 \leq t^{n-1}_z -
t^n_z < 2^{-n} $ and $t^n_{p_n} \in a_n\mathbb{Q}$.

Note that the sequence of functions $z\mapsto t^n_z$ converges uniformly to a function $z\mapsto t_z$ satisfying the required properties.

\subsection{\texorpdfstring{An annulus covering of degree $d$ that is not semiconjugate to $z^d$ acting on $S^1$}{An annulus covering of degree d that is not semiconjugate to the degree-d map on the circle}}\label{contra}

In this section we let $A=(0,1)\times S^1$ be the open annulus, and $m_d:S^1\to S^1$ be the function given by $m_d(z)=z^d$.

We denote by $\gamma_1\wedge\gamma_2$ the algebraic intersection number between two arcs in $A$ whenever it is defined.  In particular, when both arcs are loops, when one of the arcs is proper and the other is a loop or when both
arcs are defined on compact intervals but the endpoints of any of the arcs do not belong to the other arc. By convention, we set $c \wedge \gamma = 1 $ if $c: (0,1) \to A$ and $\gamma : [0,1] \to A$ satisfy:
$$c (t) = (t,1), \ \gamma (t) = (1/2, e^{2\pi i t}).$$

Note that the arc $c$ joins one end of the annulus to the other, and its  intersection number with any loop in $A$ gives the homology class of the loop.  If $\alpha$ is a loop in $A$, 
denote by $j\alpha$ the concatenation of $\alpha$ with itself $j$ times.

\begin{proposition}[\!\!\cite{iprx1}]
\label{condition}
Let $f$ be a covering of the open annulus $A$ and assume that $f$ is semiconjugated to $m_d$. Then the following condition holds:
(*) For each compact set $K\subset A$ there exists a number $C_K$ such that: given $\alpha\subset A$ a simple closed curve, $n\geq 1$ and $j\in [1, \ldots, d^{n-1}]$ then any 
$f^n$-lift $\beta$ of $j \alpha$ with endpoints in $K$ satisfies $|\beta \wedge c|\leq C_K$.
\end{proposition}

\begin{proof}
Let $h$ be the semiconjugacy between $f$ and $m_d$ and let $\tilde h$ and $\tilde f$ be the lifts of $h$ and $f$ verifying $\tilde h \tilde f=d\tilde h$.
Take $a,b\in (0,1)$  such that the set $\tilde K=[a,b]\times \R$ contains $\pi ^{-1} (K)$.
By Theorem \ref{semi} in Section \ref{ac}, there exists a constant $M$ such that $|\tilde h(x,y)-y|\leq M$ whenever $(x,y)\in \tilde K$.

Take $\alpha$, $n$, $j$ and $\beta$ as in the statement. Let $\tilde\beta$ be a lift of
$\beta$ to the universal covering. As the endpoints of $\beta$ belong to $K$, the extreme points $(x_1,y_1)$ and $(x_2,y_2)$ of $\tilde \beta$ belong to $\tilde K$.
Note that it is enough to show that $|y_2-y_1|$ is bounded by a constant $C_K$. We will prove that this holds with $C_K=2M+1$.

Note that $\tilde f^n(x_1,y_1)$ and $\tilde f^n(x_2,y_2)$ are the endpoints of a lift of $j\alpha$ to the universal covering. This means that $|\tilde f^n(x_1,y_1)-\tilde f^n(x_2,y_2)|=(0,j)$. It follows
that
\[
    |\tilde h (\tilde f^n(x_1,y_1))-\tilde h (\tilde f^n(x_2,y_2))|=j.
\]
Then,
$$
|d^n\tilde h (x_1,y_1)-d^n\tilde h (x_2,y_2)|=|\tilde h (\tilde f ^n(x_1,y_1))-\tilde h (\tilde f^n(x_2,y_2))|=j\leq d^n,
$$
so $|\tilde h (x_1,y_1)-\tilde h(x_2,y_2)|\leq 1$.
Finally, using that the endpoints of $\beta$ belong to $K$, it follows that
\[
|y_1-y_2|\leq |y_1-\tilde h (x_1,y_1)|+|\tilde h (x_1,y_1)-\tilde h (x_2,y_2)|+|\tilde h (x_2,y_2)-y_2|\leq 2M+1.\qedhere
\]
\end{proof}

Now we will construct $f$, a covering of the open annulus for which the condition (*) introduced in the previous proposition does not hold, and therefore $f$ is not semiconjugate
to $m_d$.

Let $\{a_n :\ n\in \Z\}$ be an increasing sequence of positive real numbers such that $a_n\to 0$ when $n\to-\infty$ and $a_n\to 1$ when $n\to+\infty$.
Define the annuli $A_n$ as the product $[a_n,a_{n+1}]\times S^1$, for each $n\in \Z$.
Let also $\lambda_n$ be the affine increasing homeomorphism carrying $[0,1]$ onto $[a_n,a_{n+1}]$.
Define $f(x,z)=(\lambda_{n+1}\lambda_n^{-1}x,z^2)$ for every $n\leq -1$, that is, $(x,z)\in \cup_{n<0}A_n$.

Assume $f$ constructed until the annulus $A_{n-2}$ for some $n$ and we will show how to construct the restriction of $f$ to $A_{n-1}$.
We will suppose that $f(a_k,z)=(a_{k+1},z^2)$ for every $k\leq n-1$ and every $z\in S^1$.

Let $\alpha$ be a curve in $A_0$ such that:

\begin{enumerate}
\item
$\alpha$ joins $(a_0,1)$ with $(a_1,1)$.
\item
The lift $\alpha_0$ of $\alpha$ to the universal covering that begins at $(a_0,0)$, ends at $(a_1,n)$.
\item
$\beta:=f^{n-1}(\alpha)$ is simple.
\end{enumerate}

Note that $f^{n-1}$ is already defined in $A_0$. To prove that such an $\alpha$ exists, take first any $\alpha'$ satisfying the first and second conditions. 
Then $f^{n-1}(\alpha')$ is a curve joining $(a_{n-1},1)$ with $(a_n,1)$. Maybe $f^{n-1}(\alpha')$ is not simple, but there exists a simple curve $\beta$
homotopic to $f^{n-1}(\alpha')$ and with the same extreme points. Then define $\alpha$ as the lift of $\beta$ under $f^{n-1}$ that begins at the point $(a_0,1)$.

Choose any simple arc $\beta'$ disjoint from $\beta$ and contained in $A_{n-1}$, joining the points $(a_{n-1},-1)$ and $(a_n,-1)$.
Note that $f^{-(n-1)}(\beta')$ is the union of $2^{n-1}$ curves all disjoint from $\alpha$. Choose any one of these curves and denote it 
$\alpha'$. Note that it does not intersect $\alpha$.
Note also that there is a lift $\alpha'_0$ of $\alpha'$ that begins in a point $(a_0,t)$ and ends at $(a_1,n+t)$ in the universal covering. Now,  choose a
point $Y\in\alpha'$ whose lift $Y'$ in $\alpha'_0$ has second coordinate greater than $n$, and a point $X$ in $\alpha$ whose lift $X'$ in $\alpha_0$ has second
coordinate less than $1/2$.

Note that $f^{n-1}(X)\in\beta$ and $f^{n-1}(Y)\in\beta'$. The complement of $\beta\cup\beta'$ in the interior of $A_{n-1}$ consists of two open discs, each one of which is homeomorphic
to the complement of $s$ in the interior of $A_n$, where $s$ is the segment $\{(x,1)\ :\ a_n<x<a_{n+1}\}.$  So,  it is possible to take a homeomorphism from each of these components
and extend it to the boundary in such a way that the image of $\beta$ is $s$ and the image of $\beta'$ is also $s$, and carrying $X$ and $Y$ to the same point $p\in s$.
If the homeomorphisms are taken carefully, they induce a covering $f$ from $A_{n-1}$ to $A_n$.
Now take a simple closed curve $\gamma$ contained in $A_n$ and with base point $p$. Note that if $j=2^{n}/2$, then the curve $j\gamma$ lifts under $f^n$ to a curve joining $X$ to $Y$.
But the difference between the second coordinates of $Y'$ and $X'$ is greater than $n-1$. It follows that the intersection number of a lift of $j\gamma$ and the arc
$c$ in $A_0$ exceeds $n-1$.

As this can be done for every positive $n$, taking $K=A_0$ in Proposition \ref{condition}, it follows that $f$ does not satisfy condition (*).

\section{Proof of Brouwer's theorem}\label{newbrou}

This paper contains this one tiny original piece of work: a new proof of  Brouwer's fixed point theorem for orientation
preserving plane homeomorphisms.  Actually, of the part that says that period $2$ implies fixed point.

In the 80's Fathi wrote an amazing paper  \cite{fathi}, where he gave an easy-to-understand proof of the following:

\begin{theorem}[\!\!\cite{brou}] Let $f$ be an orientation preserving plane homeomorphism with a non-wandering point.  Then it has a fixed point.
\end{theorem}

More importantly, he reduced the proof of that theorem to proving the following:

\begin{lemma}\label{new}  Let $f:S^2\to S^2$ be an orientation preserving homeomorphism.  If $f$ has a periodic point of period $2$, then it has at least two fixed points.
\end{lemma}

The purpose of this section is to give an even simpler proof of this previous lemma using Lemma \ref{indice} in Section \ref{Lefschetz}.

\begin{proof} Let $p\in S^2$ such that $f(p)=q$, $f(q)=p$ and let $A=S^2\backslash \{p,q\}$.  We identify $A$ with $S^1\times (0,1)$ and its universal covering
 $\tilde A$ with $\R\times (0,1)$.  Note that as $f$ preserves orientation, $F(x+(1,0))=F(x)-(1,0)$ for any lift $F:\tilde A\to \tilde A$.
 
 Note also that it is enough to show that any lift $F:\tilde A\to \tilde A$ of $f$ has a fixed point.  Indeed, if $F(x)=x$ and $(F+(1,0))(w)=w$, an easy computation shows that $x$ and $w$ project to
different points in $A$ (all of this is already contained in Fathi's paper).  Now, to see that any lift $F:\tilde A\to \tilde A$ of $f$ has a fixed point, we use Lemma \ref{indice}
in Section \ref{Lefschetz} to obtain a simple closed curve $\Gamma$ such that  $i(\Gamma, F)=1$. The lines $\alpha$ and $\beta$ can be taken to be the lifts of two simple closed loops around $p$ and $q$ respectively, and the lines $\gamma$ and
$\delta$ can be chosen as proper lines connecting both ends of $\tilde A$ that are sufficiently far to the left and right respectively.  Now, $f(p)=q$, $f(q)=p$ and 
$F(x+(1,0))=F(x)-(1,0)$ imply that the hypotheses of the lemma are satisfied, and therefore $i(\Gamma, F)=1$.  Now, by Lemma \ref{aldo}, $F$ has a fixed point.
\end{proof}

\section{Concluding remarks}

This survey shows that we are still far from understanding the precise obstructions preventing a map $f:S^2\to S^2$ from satisfying the growth rate inequality. Instead, 
we currently have an intriguing assortment of methods, ideas, and results that have yet to align into a coherent whole.

A central problem is understanding fixed point properties on indecomposable continua, a question as old as the field itself. Equally important is determining whether (and how) dynamics defined on 
such continua can be extended to their ambient spaces.

Finally, it would be very interesting to establish a method to link the theory of plane homeomorphisms to general sphere endomorphisms, like it was done
naturally with parabolic orbifolds.


{\small
\begin{thebibliography}{10}
    \bibitem{barcor}
    H.~Barge and L.~Hern{\'a}ndez-Corbato.
    \newblock On the growth rate inequality for self-maps of the sphere.
    \newblock {\em Indiana University Mathematics Journal}, 74(4):1007--1022, 2025.
    
    \bibitem{bam}
    M.~Barge and B.~F. Martensen.
    \newblock Classification of expansive attractors on surfaces.
    \newblock {\em Proceedings of the American Mathematical Society},
      45(5):1075--1085, 2013.
    
    \bibitem{bama}
    M.~Barge and J.~Martin.
    \newblock The construction of globarl attractors.
    \newblock {\em Ergodic Theory and Dynamical Systems}, 31(6):1619--1639, 2011.
    
    \bibitem{bm}
    M.~Bonk and D.~Meyer.
    \newblock {\em Expanding Thurston maps}.
    \newblock American Mathematical Society, 2017.
    
    \bibitem{boronski}
    J.~P. Boro{\'n}ski.
    \newblock A fixed point theorem for the pseudo-circle.
    \newblock {\em Topology and its Applications}, 158(6):775--778, 2011.
    
    \bibitem{jernej}
    J.~P. Boronski, J.~Cinc, and P.~Oprocha.
    \newblock Beyond {$0$} and {$\infty$}: a solution to the barge entropy
      conjecture.
    \newblock To appear in TAMS, 2025.
    
    \bibitem{brou}
    L.~E.~J. Brouwer.
    \newblock Beweis des ebenen translations\-satzes.
    \newblock {\em Math. Ann.}, 72:37--54, 1912.
    
    \bibitem{bfgj}
    R.~F. Brown, M.~Furi, L.~Gorniewicz, and B.~Jiang.
    \newblock {\em Handbook of Topological Fixed Point Theory}.
    \newblock Springer Netherlands, 2005.
    
    \bibitem{sole}
    J.~Charatonik and P.~Pellicer~Covarrubias.
    \newblock On covering maps on solenoids.
    \newblock {\em Proceedings of the American Mathematical Society},
      130(7):2145--2154, 2001.
    
    \bibitem{cmr}
    D.~Childers, J.~C. Mayer, and J.~T. Rogers~Jr.
    \newblock Indecomposable continua and the julia sets of polynomials, ii.
    \newblock {\em Topology and its Applications}, 153:1593--1602, 2006.
    
    \bibitem{dh}
    A.~Douady and J.~H. Hubbard.
    \newblock A proof of thurston's topological characterization of rational
      functions.
    \newblock {\em Mathematische Zeitschrift}, 284:209--229, 2016.
    
    \bibitem{fathi}
    A.~Fathi.
    \newblock An orbit closing proof of brouwer's lemma on translation arcs.
    \newblock {\em Enseign. Math.}, 33:315--322, 1987.
    
    \bibitem{gam}
    K.~Gammon.
    \newblock Lifting croocked circular chain to covering spaces.
    \newblock {\em Topolog Proceedings}, 36:1--10, 2010.
    
    \bibitem{elena}
    E.~Gomes.
    \newblock Indecomposable continua and the julia sets of pseudo-polynomials.
    \newblock {\em Discrete and Continuous Dynamical Systems}, 45(8):2471--2484,
      2025.
    
    \bibitem{gmnp2}
    G.~Graff, M.~Misiurewicz, and P.~Nowak-Przygodzki.
    \newblock Shub's conjecture for smooth longitudinal maps of {$S^m$}.
    \newblock {\em Journal of Difference Equations and Applications},
      24(7):1044--1054, 2018.
    
    \bibitem{gmnp}
    G.~Graff, M.~Misiurewicz, and P.~Nowak-Przygodzki.
    \newblock Periodic points for sphere maps preserving monopole foliations.
    \newblock {\em Qualitative Theory of Dynamical Systems}, 18:533--546, 2019.
    
    \bibitem{handel}
    M.~Handel.
    \newblock A pathological area preserving {$C^\infty$} diffeomorphism of the
      plane.
    \newblock {\em Proceedings of the American Mathematical Society},
      86(1):163--168, 1982.
    
    \bibitem{heath}
    J.~W. Heath.
    \newblock Weakly confluent, 2-to-1 maps on hereditarily indecomposable
      continua.
    \newblock {\em Proceedings of the American Mathematical Society},
      117(2):569--573, 1993.
    
    \bibitem{lhc}
    L.~Hernandez-Corbato and F.~Ruiz~del Portal.
    \newblock Fixed point indices of planar continuous maps.
    \newblock {\em Discrete and Continuous Dynamical Systems}, 35(7):2979--2995,
      2015.
    
    \bibitem{mis}
    G.~Honorato, J.~Iglesias, A.~Portela, A.~Rovella, F.~Valenzuela, and J.~Xavier.
    \newblock On the growth rate inequality for periodic points in the two sphere.
    \newblock {\em Journal of Difference Equations and Applications},
      25(2):219--232, 2019.
    
    \bibitem{covth}
    J.~Iglesias, A.~Portela, A.~Rovella, and J.~Xavier.
    \newblock Non-locally connected spaces i: Coverings and branched coverings on
      metric spaces.
    \newblock Preprint.
    
    \bibitem{iprx2}
    J.~Iglesias, A.~Portela, A.~Rovella, and J.~Xavier.
    \newblock Dynamics of annulus maps ii: Periodic points for coverings.
    \newblock {\em Fundamenta Mathematicae}, 235:257--276, 2016.
    
    \bibitem{iprx3}
    J.~Iglesias, A.~Portela, A.~Rovella, and J.~Xavier.
    \newblock Dynamics of annulus maps iii: periodic points and completeness.
    \newblock {\em Nonlinearity}, 29(9):2641--2656, 2016.
    
    \bibitem{iprx1}
    J.~Iglesias, A.~Portela, A.~Rovella, and J.~Xavier.
    \newblock Dynamics of covering maps of the annulus i: semiconjugacies.
    \newblock {\em Mathematische Zeitschrift}, 284:209--229, 2016.
    
    \bibitem{bcs}
    J.~Iglesias, A.~Portela, A.~Rovella, and J.~Xavier.
    \newblock Sphere branched coverings and the growth rate inequality.
    \newblock {\em Nonlinearity}, 33(9):4613--4626, 2020.
    
    \bibitem{plig}
    J.~Iglesias, A.~Portela, A.~Rovella, and J.~Xavier.
    \newblock Branched coverings of the sphere having a completely invariant
      continuum with infinitely many wada lakes.
    \newblock {\em Topology and its Applications}, 339(B), 2023.
    
    \bibitem{thurston}
    J.~Iglesias, A.~Portela, A.~Rovella, and J.~Xavier.
    \newblock The growth rate inequality for thurston maps with non-hyperbolic
      orbifolds.
    \newblock {\em Discrete and Continuous Dynamical Systems}, 44(6):1768--1780,
      2024.
    
    \bibitem{krys}
    K.~Kuperberg.
    \newblock Fixed points of orientation reversing homeomorphisms of the plane.
    \newblock {\em Proceedings of the American Mathematical Society}, 112:223--229,
      1991.
    
    \bibitem{Li}
    Z.~Li.
    \newblock Periodic points and the measure of maximal entropy of an expanding
      thurston map.
    \newblock {\em Transactions of the American Mathematical Society},
      368:8955--8999, 2016.
    
    \bibitem{qote}
    S.~Llavayol and J.~Xavier.
    \newblock Quotients of torus endomorphisms have parabolic orbifolds.
    \newblock {\em Conformal Geometry and Dynamics}, 28:88--96, 2024.
    
    \bibitem{milnor}
    J.~Milnor.
    \newblock {\em Dynamics of one complex variable}, volume 160 of {\em Annals of
      Mathematical Studies}.
    \newblock Princeton University Press, 2006.
    
    \bibitem{nancys}
    M.~J. Misiurewicz.
    \newblock Periodic points of latitudinal sphere maps.
    \newblock {\em Fixed Point Theory and Applications}, 16(1-2):149--158, 2014.
    
    \bibitem{prz}
    F.~Przytycki.
    \newblock Chaos after bifurcation of a morse smale diffeomorphism after a
      one-cycle saddle-node and iterations of multivalued mappings of an interval
      and a circle.
    \newblock {\em Bol. Soc. Bras. Mat.}, 18(1):29--79, 1987.
    
    \bibitem{ps}
    C.~Pugh and M.~Shub.
    \newblock Periodic points on the 2-sphere.
    \newblock {\em Discrete and Continuous Dynamical Systems}, 34(3):1171--1182,
      2014.
    
    \bibitem{s}
    M.~Shub.
    \newblock Alexander cocycles and dynamics.
    \newblock {\em Asterisque, Soci{\'e}t{\'e} Math. de France}, pages 395--413,
      1978.
    
    \bibitem{ss}
    M.~Shub and D.~Sullivan.
    \newblock A remark on the lefschetz fixed point formula for differentiable
      maps.
    \newblock {\em Topology}, 13:189--191, 1974.
    
    \bibitem{vernon}
    R.~Vernon.
    \newblock Concerning preservation of indecomposability upon taking a preimage
      under {$z\mapsto z^n$}.
    \newblock {\em Topology Proceedings}, 31(1):331--348, 2007.
\end{thebibliography}
}
\end{document}